\documentclass[12pt]{amsart}
\usepackage{amsmath,amssymb}
\usepackage{enumerate}
\allowdisplaybreaks
\title[Symmetric amenability for $\mathrm{C}^*$-algebras]{Dixmier approximation and symmetric amenability for $\mathrm{C}^*$-algebras}
\author{Narutaka Ozawa}
\address{RIMS, Kyoto University, \mbox{606-8502} Kyoto, Japan}
\email{narutaka@kurims.kyoto-u.ac.jp}
\thanks{Partially supported by JSPS (23540233)}
\subjclass{46L05; 46L10}

\keywords{Dixmier approximation, symmetric amenability, continuous bundles of von Neumann algebras}
\date{\today}
\newtheorem{thm}{Theorem}

\newtheorem{cor}[thm]{Corollary}
\newtheorem{lem}[thm]{Lemma}
\theoremstyle{definition}

\newcommand{\IB}{{\mathbb B}}
\newcommand{\IC}{{\mathbb C}}
\newcommand{\IM}{{\mathbb M}}
\newcommand{\IN}{{\mathbb N}}

\newcommand{\cO}{{\mathcal O}}
\newcommand{\cR}{{\mathcal R}}
\newcommand{\cU}{{\mathcal U}}
\newcommand{\cZ}{{\mathcal Z}}
\newcommand{\fC}{{\mathfrak C}}
\newcommand{\ru}{{\mathrm u}}

\newcommand{\e}{\varepsilon}
\newcommand{\p}{\varphi}

\DeclareMathOperator{\conv}{conv}

\DeclareMathOperator{\mul}{m}
\DeclareMathOperator{\tr}{tr}
\DeclareMathOperator{\ctr}{ctr}
\DeclareMathOperator{\Res}{Res}
\DeclareMathOperator{\ev}{ev}
\DeclareMathOperator{\Ad}{Ad}

\DeclareMathOperator{\prob}{{\mathcal M}_+^1}
\DeclareMathOperator{\Aff}{Aff}
\DeclareMathOperator{\cAff}{{\mathcal A}ff}
\newcommand{\st}{\mathrm{st}}

\newcommand{\alg}{\IC}

\begin{document}
\begin{abstract}
We study some general properties of tracial $\mathrm{C}^*$-algebras.
In the first part, we consider Dixmier type approximation theorem and
characterize symmetric amenability for $\mathrm{C}^*$-algebras.
In the second part, we consider continuous bundles of tracial
von Neumann algebras and classify some of them.
\end{abstract}
\maketitle
\section{Introduction}
The general study of tracial states on $\mathrm{C}^*$-alge\-bras has a long history,
but recently it gained a renewed interest in connection with the
ongoing classification program for finite nuclear $\mathrm{C}^*$-alge\-bras.
In this note, we record several facts about tracial $\mathrm{C}^*$-alge\-bras which
may be useful in the future study.
The results are two-fold.
First, we consider Dixmier type approximation property for $\mathrm{C}^*$-alge\-bras
and relate it to symmetric amenability.
The Dixmier approximation theorem (Theorem III.5.1 in \cite{dixmier}) states
a fundamental fact about von Neumann algebras
that for any von Neumann algebra $N$ and any element $a\in N$, the norm-closed
convex hull of $\{ uau^* : u \in \cU(N)\}$ meets the
center $\cZ(N)$ of $N$. Here $\cU(N)$ denotes the unitary group of $N$.
If $N$ is moreover a finite von Neumann algebra,
then this intersection is a singleton and consists of $\ctr(a)$.
Here $\ctr$ denotes the center-valued trace, which is the unique conditional
expectation from $N$ onto $\cZ(N)$ that satisfies $\ctr(xy)=\ctr(yx)$.
It is proved by Haagerup and Zsido (\cite{hz}) that the Dixmier approximation
theorem holds for simple $\mathrm{C}^*$-alge\-bras having at most one tracial states
(and obviously does not for simple $\mathrm{C}^*$-alge\-bras having more than one tracial states).
Recall that a $\mathrm{C}^*$-alge\-bra $A$ has the \emph{quotient tracial state property} (QTS property)
if every non-zero quotient $\mathrm{C}^*$-alge\-bra of $A$ has a tracial state (\cite{murphy}).
We denote by $T(A)$ the space of the tracial states on $A$,
equipped with the weak$^*$-topology.

\begin{thm}\label{thm:qts}
For a unital $\mathrm{C}^*$-alge\-bra $A$, the following are equivalent.
\begin{enumerate}[$(\rm i)$]
\item\label{con:qts}
The $\mathrm{C}^*$-alge\-bra $A$ has the QTS property.
\item\label{con:dix}
For every $\e>0$ and $a\in A$ that satisfy $\sup_{\tau\in T(A)}|\tau(a)|<\e$,
there are $k$ and $u_1,\ldots,u_k\in\cU(A)$ such that
$\| \frac{1}{k}\sum_{i=1}^k u_iau_i^* \|<\e$.
\end{enumerate}
\end{thm}

Unlike the case for von Neumann algebras, there is no bound of $k$ in terms of $\e$ and $\|a\|$
that works for an arbitrary element $a$ in a $\mathrm{C}^*$-alge\-bra{} (see Section \ref{sec:com}, where
we study a relation between trace zero elements and commutators).
Recall that a Banach algebra $A$ is said to be \emph{amenable} if
there is a net $(\Delta_n)_n$, called an approximate diagonal, in the algebraic
tensor product $A\otimes_\alg A$ (we reserve the symbol $\otimes$ for the minimal tensor product) such that
\begin{enumerate}[$(1)$]
\item $\sup_n\|\Delta_n\|_{\wedge}<+\infty$,
\item $(\mul(\Delta_n))_n$ is an approximate identity,
\item $\lim_n\| a\cdot\Delta_n - \Delta_n\cdot a\|_{\wedge}=0$ for every $a\in A$.
\end{enumerate}
Here $\|\,\cdot\|_{\wedge}$ is the projective norm on $A\otimes_\alg A$,
$\mul\colon A\otimes_\alg A\to A$ is the multiplication, and
$a\cdot(\sum_i x_i\otimes y_i)=\sum_i ax_i\otimes y_i$ and
$(\sum_i x_i\otimes y_i)\cdot a=\sum_i x_i\otimes y_ia$.
The celebrated theorem of Connes--Haagerup (\cite{connes,haagerup}) states
that a $\mathrm{C}^*$-alge\-bra $A$ is amenable as a Banach algebra if and only if it is nuclear.
The Banach algebra $A$ is said to be \emph{symmetrically amenable} (\cite{johnson}) if
the approximate diagonal $(\Delta_n)_n$ can be taken symmetric under the
flip $x\otimes y\to y\otimes x$. We characterize symmetric amenability
for $\mathrm{C}^*$-alge\-bras.

\begin{thm}\label{thm:symame}
For a unital $\mathrm{C}^*$-alge\-bra $A$, the following are equivalent.
\begin{enumerate}[$(\rm i)$]
\item\label{con:nucqts}
The $\mathrm{C}^*$-alge\-bra $A$ is nuclear and has the QTS property.
\item\label{con:symame0}
The $\mathrm{C}^*$-alge\-bra $A$ has an approximate diagonal $\Delta_n=\sum_{i=1}^{k(n)} x_i(n)^*\otimes x_i(n)$
such that $\lim_n\sum_{i=1}^{k(n)}\|x_i(n)\|^2=1$, $\mul(\Delta_n)=1$, and
$\lim_n\|1-\sum_{i=1}^{k(n)} x_i(n)x_i(n)^*\|=0$.
\item\label{con:symame}
The $\mathrm{C}^*$-alge\-bra $A$ is symmetrically amenable.
\item\label{con:symame1}
The $\mathrm{C}^*$-alge\-bra $A$ has a symmetric approximate diagonal $(\Delta_n)_n$ in
\[
\{ \sum_i x_i^*\otimes x_i \in A\otimes_\alg A : \sum_i \|x_i\|^2\le 1\}.
\]
\end{enumerate}
\end{thm}

Recall that a unital $\mathrm{C}^*$-alge\-bra $A$ is \emph{strongly amenable} if there is an approximate diagonal
that consists of convex combinations of $\{ u^*\otimes u : u\in\cU(A)\}$.
This property is formally stronger than symmetric amenability, but it is unclear whether
there is really a gap between these properties.

Second, we describe what is the $\mathrm{C}^*$-completion $\overline{A}^{\mathrm{u}}$ of
a unital $\mathrm{C}^*$-alge\-bra $A$ under the uniform $2$-norm.
This work is strongly influenced by the recent works of Kirchberg--R{\o}rdam (\cite{kr}),
Sato (\cite{sato}), and Toms--White--Winter (\cite{tww}), who studied the central sequence
algebra of a $\mathrm{C}^*$-alge\-bra modulo uniformly $2$-norm null sequences, in order to extend
Matui--Sato's result (\cite{ms}) from $\mathrm{C}^*$-alge\-bras with finitely many extremal tracial states
to more general ones. In fact, our result is very similar to theirs (particularly
to Kirchberg--R{\o}rdam's).
Let $A$ be a $\mathrm{C}^*$-alge\-bra and $S\subset T(A)$ be a non-empty metrizable closed face.
The reason we assume $S$ be metrizable is because it makes the description of
the boundary measures simpler.
We define the uniform $2$-norm on $A$ corresponding to $S$ by
\[
\|a\|_{2,S}=\sup\{ \tau(a^*a)^{1/2} : \tau\in S\}.
\]
The uniform $2$-norm satisfies
\[
\|ab\|_{2,S}\le\min\{\|a\|\|b\|_{2,S},\|a\|_{2,S}\|b\|\}
\ \mbox{ and }\
\sup_{\tau\in S}|\tau(a)|\le\|a\|_{2,S}.
\]
The $\mathrm{C}^*$-completion $\overline{A}^{\mathrm{u}}$ is defined to be
the $\mathrm{C}^*$-alge\-bra of the norm-bounded uniform $2$-norm Cauchy sequences,
modulo the ideal of the uniform $2$-norm null sequences.
For $\tau\in T(A)$, we denote by $\pi_\tau$ the corresponding GNS representation and
also $\|a\|_{2,\tau}=\tau(a^*a)^{1/2}$.
Let $N=(\bigoplus_{\tau\in S}\pi_\tau)(A)''$ be the enveloping
von Neumann algebra with respect to $S$.
When $S=T(A)$, it is the finite summand $A^{**}_{\mathrm{f}}$ of the second
dual von Neumann algebra $A^{**}$.
The tracial state $\tau\in S$ and the GNS representation $\pi_\tau$
extend normally on $N$.
For the center-valued trace $\ctr\colon N\to\cZ(N)$,
one has $\|a\|_{2,S}=\|\ctr(a^*a)\|^{1/2}$ and
$\overline{A}^{\mathrm{u}}$ coincides with
the closure $\overline{A}^{\st}$ of $A$ in $N$ with respect
to the strict topology associated with the Hilbert $\cZ(N)$-module $(N,\ctr)$.

Recall that the trace space $T(A)$ of a unital $\mathrm{C}^*$-alge\-bra is a Choquet simplex and so is
the closed face $S$.
We denote by $\Aff(S)$ the space of the affine continuous functions on $S$ and
consider the function system $\cAff(S)=\{ f|_{\partial S} : f\in\Aff(S)\}$
in $B(\partial S)$, where $B(\partial S)$ denotes the $\mathrm{C}^*$-alge\-bra of the bounded
Borel functions on $\partial S$.
For every $a\in A$, the formula $\hat{a}(\tau)=\tau(a)$ defines
a function $\hat{a}$ in $\Aff(S)$ (or $\cAff(S)$).
We note that $\{\hat{a} : a\in A\}$ is dense in $\Aff(S)$ (in fact equal, see \cite{cp}).
Let $\prob(\partial S)$ be the space of the probability measures on
the extreme boundary $\partial S$ of $S$.
Since $S$ is a metrizable Choquet simplex, every $\tau\in S$ has
a unique representing measure $\mu_\tau\in\prob(\partial S)$, which satisfies
\[
\tau(a)=\int \lambda(a)\,d\mu_\tau(\lambda) = \int \hat{a}(\lambda)\,d\mu_\tau(\lambda)
\]
for every $a\in A$ (Theorem II.3.16 in \cite{alfsen}).
The center $\cZ(\cAff(S))$ is defined to be
\[
\cZ(\cAff(S))=\{ f\in B(\partial S) : f\cAff(S)\subset\cAff(S)\}\subset \cAff(S).
\]
When $\partial S$ is closed (i.e., when $S$ is a Bauer simplex), one has $\cAff(S)=C(\partial S)$
and $\cZ(\cAff(S))=C(\partial S)$. However in general, the center $\cZ(\cAff(S))$ can be
trivial (see Section II.7 in \cite{alfsen}).

\begin{thm}\label{thm:ast}
Let $A$, $S$, and $N$ be as above. Then, there is a unital $*$-homo\-mor\-phism
$\theta\colon B(\partial S)\to\cZ(N)$ with ultra\-weakly dense range such that
$\theta(\hat{a})=\ctr(a)$ and
\[
\tau(\theta(f)a) = \int f(\lambda)\lambda(a)\,d\mu_\tau(\lambda) = \int f\hat{a}\,d\mu_\tau
\]
for every $a\in A$ and $\tau\in S$.
One has
\[
\overline{A}^{\st}=\{ x\in N : \ctr(xA)\subset \theta(\cAff(S)),\ \ctr(x^*x)\in \theta(\cAff(S))\}.
\]
In particular,
\[
\overline{A}^{\st}\cap\cZ(N) = \{ \theta(f) : f \in \cZ(\cAff(S))\}.
\]
Moreover, if $\partial S$ is closed, then for every $\tau\in\partial S$,
one has $\pi_\tau(\overline{A}^{\st})=\pi_\tau(N)=\pi_\tau(A)''$.
\end{thm}

Takesaki and Tomiyama (\cite{tt}) have studied the structure of a $\mathrm{C}^*$-alge\-bra, for
which the set of pure states is closed in the state space, by using a continuous
bundle of $\mathrm{C}^*$-alge\-bras (see also \cite{fell}). We carry out a similar study in Section \ref{sec:ax}
for a $\mathrm{C}^*$-alge\-bra $A$, for which $\partial S$ is closed, in terms of a continuous $\mathrm{W}^*$-bundle,
and present $\mathrm{W}^*$-analogues of a few results for
$\mathrm{C}^*$-bundles obtained in \cite{hrw,dw}.
In particular, we give a criterion for a continuous $\mathrm{W}^*$-bundle
over a compact space $K$ with all fibers isomorphic to the hyperfinite
$\mathrm{II}_1$ factor $\cR$ to be isomorphic to the trivial bundle $C_\sigma(K,\cR)$,
the $\mathrm{C}^*$-alge\-bra of the norm-bounded and ultra\-strongly continuous functions
from $K$ into $\cR$. We denote the evaluation map at $\lambda\in K$
by $\ev_\lambda\colon C_\sigma(K,\cR)\to\cR$. As an application, we show
that $\overline{A}^{\st}\cong C_\sigma(\partial S,\cR)$ for certain $A$.

\begin{thm}\label{thm:stct}
Let $A$ be a separable $\mathrm{C}^*$-alge\-bra and $S\subset T(A)$ be a closed face.
Assume that $\pi_\tau(A)''\cong\cR$ for all $\tau\in\partial S$
and that $\partial S$ is a compact space with finite covering dimension.
Then, one can coordinatize the isomorphisms $\pi_\tau(A)''=\cR$ in such a way that
they together give rise to a $*$-homo\-mor\-phism
$\pi\colon A\to C_\sigma(\partial S,\cR)$ such that
$\pi_\tau=\ev_\tau\circ\pi$. The image of $\pi$ is
dense with respect to the uniform $2$-norm.
\end{thm}

\subsection*{Acknowledgment.}
The author is grateful to Yasuhiko Sato for many illuminating conversations.
The author was partially supported by JSPS (23540233).

\section{QTS property and symmetric amenability}

\begin{proof}[Proof of Theorem~\ref{thm:qts}]
Ad $(\rm{\ref{con:qts}})\Rightarrow({\rm\ref{con:dix}})$.
Although the proof becomes a bit shorter if we use Theorem 5 in \cite{hz},
we give here a more direct proof of this implication.
Let $a\in A$ and $\e>0$ be given as in condition $({\rm\ref{con:dix}})$.
Let $\e_0=\sup_{\tau\in T(A)}|\tau(a)|<\e$.
We decompose the second dual von Neumann algebra $A^{**}$
into the finite summand $A^{**}_{\mathrm{f}}$ and the properly infinite
summand $A^{**}_{\infty}$. We denote the corresponding embedding of $A$
by $\pi_{\mathrm{f}}$ and $\pi_\infty$, and the center-valued
trace of $A^{**}_{\mathrm{f}}$ by $\ctr$.
We note that $\|\ctr(\pi_{\mathrm{f}}(a))\|=\e_0$.
By the Dixmier approximation theorem, there
are $v_1,\ldots,v_k\in \cU(A^{**}_{\mathrm{f}})$ such that
$\|\ctr(\pi_{\mathrm{f}}(a)) - \frac{1}{k}\sum_{i=1}^k v_i\pi_{\mathrm{f}}(a)v_i^*\|<\e-\e_0$.
On the other hand, by Halpern's theorem (\cite{halpern}),
there are $w_1,\ldots,w_l\in \cU(A^{**}_{\infty})$ such that
$\|\frac{1}{l}\sum_{j=1}^l w_i\pi_{\infty}(a)w_i^*\|<\e$.
Before giving the detail of the proof of this fact, we finish the proof of
$(\rm{\ref{con:qts}})\Rightarrow({\rm\ref{con:dix}})$.
By allowing multiplicity, we may assume that $k=l$ and consider $u_i=v_i\oplus w_i\in A^{**}$.
Then, $\|\frac{1}{k}\sum_{i=1}^k u_iau_i^*\|<\e$ in $A^{**}$.
For each $i$, take a net $(u_i(\lambda))_\lambda$ of unitary
elements in $A$ which converges to $u_i\in A^{**}$
in the ultra\-strong$^*$-topology. By the Hahn-Banach theorem,
$\conv\{\frac{1}{k}\sum_{i=1}^k u_i(\lambda)au_i(\lambda)^*\}_\lambda$ contains an element
of norm less than $\e$.

Now, we explain how to apply Halpern's theorem. Let $Z$ (resp.\ $I$) be
the center (resp.\ strong radical) of $A^{**}_{\infty}$.
Let $\Lambda$ be the directed set of all finite partitions of unity by projections in $Z$,
and $\lambda=\{p_{\lambda,i}\}_i\in\Lambda$ be given.
Applying the QTS property to the non-zero $*$-homomorphism
$A\ni x\mapsto p_{\lambda,i}(\pi_\infty(x)+I)\in p_{\lambda,i}((\pi_\infty(A)+I)/I)$,
one obtains a (tracial) state $\tau_{\lambda,i}$ on $\pi_\infty(A)+I$
such that $\tau_{\lambda,i}(p_{\lambda,i})=1$, $\tau_{\lambda,i}(I)=0$,
and $|\tau_{\lambda,i}(\pi_\infty(a))|\le\e_0$.
Let $\tilde{\tau}_{\lambda,i}$ be a state extension of it on $p_{\lambda,i}A^{**}_{\infty}$.
We define the linear map $\p_\lambda\colon A^{**}_{\infty}\to Z$ by
$\p_\lambda(x)=\sum_i\tilde{\tau}_{\lambda,i}(x)p_{\lambda,i}$, and
take a limit point $\p\colon A^{**}_{\infty}\to Z$. The map $\p$
is a unital positive $Z$-linear map such that $\p(I)=0$ and $\|\p(\pi_\infty(a))\|\le\e_0$.
By Halpern's theorem (Theorem 4.12 in \cite{halpern}), the norm-closed convex hull
of the unitary conjugations of $\pi_\infty(a)$ contains $\p(\pi_\infty(a))$.

Ad $({\rm\ref{con:dix}})\Rightarrow(\rm{\ref{con:qts}})$.
Suppose that there is a closed two-sided proper ideal $I$ in $A$ such
that $A/I$ does not have a tracial state. Let $e_n$ be the approximate unit of $I$.
Then, one has $\tau(1-e_n)\searrow 0$ for every $\tau\in T(A)$. By Dini's theorem,
there is $n$ such that $q=1-e_n$ satisfies $\tau(q)<1/2$ for all $\tau\in T(A)$.
By condition $({\rm\ref{con:dix}})$, there are $u_1,\ldots,u_k\in\cU(A)$ such that
$\|\frac{1}{k}\sum_{i=1}^k u_iqu_i^*\|<1/2$, which is in contradiction with the fact that
$\frac{1}{k}\sum_{i=1}^k u_iqu_i^*\in 1+I$.
\end{proof}

\begin{proof}[Proof of Theorem~\ref{thm:symame}]
The implication $({\rm\ref{con:symame1}})\Rightarrow(\rm{\ref{con:symame}})$
is obvious and $({\rm\ref{con:symame}})\Rightarrow(\rm{\ref{con:nucqts}})$
is standard: Since amenability implies nuclearity by Connes's theorem (\cite{connes}),
we only have to prove the QTS property. Let $(\Delta_n)_n$ be a symmetric approximate diagonal
and define $\mul_\Delta(a)=\sum_i x_iay_i$
for $\Delta=\sum_i x_i\otimes y_i\in A\otimes_\alg A$ and $a\in A$.
Then, for any proper ideal $I$ in $A$ and a state $\p$ on $A$ such that $\p(I)=0$,
any limit point $\tau$ of $(\p\circ\mul_{\Delta_n})_n$ is a bounded trace
such that $\tau(I)=0$ and $\tau(1)=1$. By polar decomposition,
one obtains a tracial state on $A$ which vanishes on $I$.

We prove the implication
$({\rm\ref{con:nucqts}})\Rightarrow(\rm{\ref{con:symame0}})\Rightarrow(\rm{\ref{con:symame1}})$.
Since $A$ is nuclear, it is amenable thanks to Haagerup's theorem (Theorem 3.1 in \cite{haagerup}).
Moreover, there is an approximate diagonal $(\Delta'_n)_n$ in the convex
hull of $\{ x^*\otimes x : \|x\|\le1\}$.
We note that $\e_n:=\| 1 - \mul(\Delta'_n) \|\to0$.
We fix $n$ for the moment and write $\Delta'_n=\sum_i x_i^*\otimes x_i$.
By replacing $x_i$ with $x_i\mul(\Delta'_n)^{-1/2}$, we may assume $\mul(\Delta'_n)=1$
but $\sum_i \|x_i\|^2\le(1-\e_n)^{-1}$.
Since $\tau(\sum_i x_ix_i^*)=1$ for all $\tau\in T(A)$, Theorem~\ref{thm:qts}
provides $u_1,\ldots,u_l\in\cU(A)$ such that
$\|\frac{1}{l}\sum_{j=1}^l \sum_i u_jx_ix_i^*u_j^*\|\le 1+\e_n$. Thus,
$\Delta_n=\frac{1}{l}\sum_{i,j} x_i^*u_j^*\otimes u_jx_i$ satisfies condition $(\rm{\ref{con:symame0}})$.
Now, rewrite $\Delta_n$ as $\sum_i y_i^*\otimes y_i$. Then,
$\Delta^{\sharp}_n=(\sum_i\|y_i\|^2)^{-2}\sum_{i,j} y_i^*y_j\otimes y_j^*y_i$ is
a symmetric approximate diagonal that meets condition $(\rm{\ref{con:symame1}})$.
\end{proof}

\section{Trace zero elements and commutators}\label{sec:com}

In this section, we consider the trace zero elements in a $\mathrm{C}^*$-alge\-bra.
A simple application of the Hahn--Banach theorem
implies that $a\in A$ satisfies $\tau(a)=0$ for all $\tau\in T(A)$
if and only if it belongs to the norm-closure of the subspace $[A,A]$
spanned by commutators $[b,c]=bc-cb$, $b,c\in A$.
Moreover, such $a$ can be written as a convergent sum of commutators (\cite{cp}).
There are many works as to how uniformly this happens (\cite{pp,fack,fh,marcoux,pop} just to name a few).
The following fact is rather standard.
\begin{thm}\label{thm:com}
There is a constant $C>0$ which satisfies the following.
Let $A$ be a $\mathrm{C}^*$-alge\-bra and $a\in A$ and $\e>0$ be such that $\sup_{\tau\in T(A)}|\tau(a)|<\e$.
Then, there are $k\in\IN$ and $b_i$ and $c_i$ in $A$ such that
$\sum_{i=1}^k\|b_i\|\|c_i\|\le C\|a\|$ and $\|a-\sum_{i=1}^k[b_i,c_i]\|<\e$.
\end{thm}

Unlike the case for von Neumann algebras, there is no bound on $k$ in terms of $\e$ and $\|a\|$
that works for general $\mathrm{C}^*$-alge\-bras. A counterexample is constructed by Pedersen and Petersen
(Lemma 3.5 in \cite{pp}: the element $x_n-y_n\in [A_n,A_n]$ constructed there has the property
that $\| (x_n-y_n) -z \|\geq1$ for any sum $z$ of $n$ self-commutators). This also means that $k$
in Theorem~\ref{thm:qts} depends on the particular element $a$ in $A$.
Nevertheless one can bound $k$ under some regularity condition.
Recall that $A$ is said to be \emph{$\cZ$-stable} if
$A\cong\cZ\otimes A$ for the Jiang--Su algebra $\cZ$.
The Jiang--Su algebra $\cZ$ is a simple $\mathrm{C}^*$-alge\-bra which is
an inductive limit of prime dimension drop algebras
and such that $\cZ\cong\cZ^{\otimes\infty}$ (Theorem 2.9 and Theorem 4 in \cite{js}).

\begin{thm}\label{thm:comreg}
There is a constant $C>0$ which satisfies the following.
Let $A$ be an exact $\cZ$-stable $\mathrm{C}^*$-alge\-bra, and
$\e>0$ and $a\in A$ be such that $\sup_{\tau\in T(A)}|\tau(a)|<\e$.
Then, for every $R\in\IN$, there are $b(r)$ and $c(r)$ in $A$ such that
$\sum_{r=1}^R\|b(r)\|\|c(r)\|\le C\|a\|$ and $\|a-\sum_{r=1}^R[b(r),c(r)]\|<\e+C\|a\|R^{-1/2}$.
\end{thm}

\begin{proof}[Proof of Theorem~\ref{thm:com}]
Let $a\in A$.
We denote by $\ctr$ the center-valued trace from the second dual von Neumann algebra
$A^{**}$ onto the center $\cZ(A^{**}_{\mathrm{f}})$ of the finite summand $A^{**}_{\mathrm{f}}$
of $A^{**}$.
One has $\| \ctr(a) \| = \sup_{\tau\in T(A)}|\tau(a)|<\e$ and $a' := a-\ctr(a)$ has
zero traces. By a theorem of Fack and de la Harpe, for $C=2\cdot12^2$ and $m=10$,
there are $b_i,c_i\in A^{**}$ such that $\sum_{i=1}^m\|b_i\|\|c_i\|\le C\|a\|$ and
$ a' = \sum_{i=1}^m [b_i,c_i]$. See \cite{marcoux,pop} for a better estimate of
$C$ and $m$. By Kaplansky's density theorem, there is a net $(b_i(\lambda))_\lambda$
in $A$ such that $\|b_i(\lambda)\|\le\|b_i\|$ and $b_i(\lambda)\to b_i$ ultra\-strongly.
Likewise for $(c_i(\lambda))_\lambda$.
Since
\[
\| \lim_\lambda ( a - \sum_{i=1}^m [b_i(\lambda),c_i(\lambda)] ) \| = \| a - a' \| <\e,
\]
there is $a''\in\conv \{\sum_{i=1}^m [b_i(\lambda),c_i(\lambda)] \}_\lambda$
which satisfies $\|a-a''\|<\e$.
\end{proof}

The proof of Theorem~\ref{thm:comreg} is inspired by \cite{haagerup:new} and uses
the free semicircular system and random matrix argument of
Haagerup--Thorbj{\o}rnsen (\cite{ht}).
Let $\cO_\infty$ be the Cuntz algebra generated by isometries $l_i(r)$
such that $l_i(r)^*l_j(s)=\delta_{i,j}\delta_{r,s}$, and let
$S_i(r) := l_i(r)+l_i(r)^*$ be the corresponding semicircular system.
We note that $\fC:=\mathrm{C}^*(\{S_i(r) : i,r\})$ is $*$-isomorphic to the reduced free
product of the copies of $C([-2,2])$ with respect to the Lebesgue measure
(see Section 2.6 in \cite{vdn}), and the corresponding tracial state coincides
with the restriction of the vacuum state on $\cO_\infty$ to $\fC$.

\begin{lem}\label{lem:ss}
Let $b_i,c_i\in A$ be such that $\|b_i\|=\|c_i\|$.
Then, for every $R\in\IN$, letting
$\tilde{b}(r)=\sum_{i=1}^n S_i(r)\otimes b_i$
and $\tilde{c}(r)=\sum_{j=1}^n S_j(r)\otimes c_j$,
one has
\[
\frac{1}{R}\sum_{r=1}^R \|\tilde{b}(r)\|\|\tilde{c}(r)\|\le 4\sum\|b_i\|\|c_i\|
\]
and
\[
\| 1\otimes\sum_{i=1}^n[b_i,c_i] - \frac{1}{R}\sum_{r=1}^R [\tilde{b}(r),\tilde{c}(r)]\|
\le\frac{6}{\sqrt{R}}\sum_i\|b_i\|\|c_i\|.
\]
\end{lem}
\begin{proof}
For every $r$, one has
\begin{align*}
\| \tilde{b}(r)\|
&\le \|\sum l_i(r)\otimes b_i\|+\|\sum l_i(r)^*\otimes b_i\|\\
&= \|\sum b_i^*b_i\|^{1/2}+\|\sum b_ib_i^*\|^{1/2}
\le 2(\sum\|b_i\|^2)^{1/2},
\end{align*}
and likewise for $\tilde{c}(r)$.
It follows that $\|\tilde{b}(r)\|\|\tilde{c}(r)\|\le4\sum\|b_i\|\|c_i\|$.
Moreover,
\begin{align*}
\tilde{b}(r)\tilde{c}(r)
 = \sum_{i,j}(\delta_{i,j}1+l_i(r)l_j(r)+l_i(r)l_j(r)^*+l_i(r)^*l_j(r)^*)\otimes b_ic_j,
\end{align*}
and
\begin{align*}
\|\sum_{r,i,j}l_i(r)l_j(r)\otimes b_ic_j\|
&= \|\sum_{r,i,j} c_j^*b_i^*b_ic_j\|^{1/2}
\le R^{1/2}\sum_i\|b_i\|\|c_i\|,\\
\|\sum_{r,i,j}l_i(r)^*l_j(r)^*\otimes b_ic_j\|
&= \|\sum_{r,i,j} b_ic_jc_j^*b_i^*\|^{1/2}
\le R^{1/2}\sum_i\|b_i\|\|c_i\|,\\
\|\sum_{r,i,j}l_i(r)l_j(r)^*\otimes b_ic_j\|
&= \max_r \| \sum_{i,j}l_i(r)l_j(r)^*\otimes b_ic_j\|
\le \sum_i\|b_i\|\|c_i\|.
\end{align*}
Likewise for $\tilde{c}(r)\tilde{b}(r)$, and one obtains the conclusion.
\end{proof}

\begin{proof}[Proof of Theorem~\ref{thm:comreg}]
Let $a\in A\setminus\{0\}$ be such that $\sup_{\tau\in T(A)}|\tau(a)|<\e$.
Since $\cZ\cong\cZ^{\otimes\infty}$, we may assume that $A=\cZ\otimes A_0$ and $a\in A_0$.
By Theorem~\ref{thm:com}, there are $b_i,c_i$ such that $\|b_i\|=\|c_i\|$,
$\sum_{i=1}^k\|b_i\|\|c_i\|\le C\|a\|$, and $\| a - \sum_{i=1}^k[b_i,c_i]\|<\e$.
Recall the theorem of Haagerup and Thorbj{\o}rnsen (\cite{ht}) which states
that the $\mathrm{C}^*$-alge\-bra $\fC$ can be embedded into $\prod \IM_n/\bigoplus\IM_n$.
By exactness of $A_0$, there is a canonical $*$-iso\-mor\-phism
\[
(\prod \IM_n/\bigoplus\IM_n)
\otimes A_0 \cong ((\prod \IM_n)\otimes A_0)/(\bigoplus\IM_n\otimes A_0).
\]
Lemma~\ref{lem:ss}, combined with this fact, implies that
there are matrices $s^{(n)}_i(r)\in\IM_n$ such that
$\tilde{b}^{(n)}(r)=\sum_{i=1}^k s^{(n)}_i(r)\otimes b_i$
and $\tilde{c}^{(n)}(r)=\sum_{j=1}^k s^{(n)}_j(r)\otimes c_j$
satisfy
\[
\limsup_n\frac{1}{R}\sum_{r=1}^R \|\tilde{b}^{(n)}(r)\|\|\tilde{c}^{(n)}(r)\|
 \le 4\sum\|b_i\|\|c_i\|\le 4C\|a\|
\]
and
\[
\limsup_n\| 1\otimes a - \frac{1}{R}\sum_{r=1}^R [\tilde{b}^{(n)}(r),\tilde{c}^{(n)}(r)]\|
\le\e+\frac{6C\|a\|}{\sqrt{R}}.
\]
For every relatively prime $p,q\in\IN$, the Jiang--Su algebra $\cZ$ contains
the prime dimension drop algebra
\[
I(p,q)=\{ f \in C([0,1],\IM_p\otimes\IM_q) : f(0)\in\IM_p\otimes\IC1,\,f(1)\in\IC1\otimes\IM_q\}
\]
and hence $t\IM_q$ and $(1-t)\IM_p$ also, where $t\in I(p,q)$ is the identity function on $[0,1]$.
It follows that there are $b(r),c(r),b'(r),c'(r)\in \cZ\otimes A_0$ such that
\[
\frac{1}{R}\sum_{r=1}^R (\| b(r)\|\| c(r)\|+\| b'(r)\|\| c'(r)\|)< 9C\|a\|
\]
and
\[
\| a - \frac{1}{R}\sum_{r=1}^R ([b(r),c(r)]+[b'(r),c'(r)])\|
<\e+\frac{7C\|a\|}{\sqrt{R}}.
\]
Here, we note that $\|t\otimes x+(1-t)\otimes y\|=\max\{\|x\|,\|y\|\}$ for any $x$ and $y$.
\end{proof}

Let $(A_n)_n$ be a sequence of $\mathrm{C}^*$-alge\-bras and $\cU$ be a free ultra\-filter on $\IN$.
We denote by
\[
\prod A_n =\{ (a_n)_{n=1}^\infty : a_n\in A_n,\,\sup_n\|a_n\|<+\infty\}
\]
the $\ell_\infty$-direct sum of $(A_n)$, and by
\[
\prod A_n/\cU = (\prod A_n) / \{ (a_n)_{n=1}^\infty : {\lim}_{\cU}\|a_n\|=0\}
\]
the ultra\-product of $A_n$.
For every $m$, we view $\tau \in T(A_m)$ as an element of $T(\prod A_n)$
by $\tau((a_n)_n)=\tau(a_m)$. For each $(\tau_n)_n\in\prod T(A_n)$,
there is a corresponding tracial state
$\tau_\cU := {\lim}_{\cU}\tau_n$ on $\prod A_n/\cU$, defined by
\[
\tau_{\cU}((a_n)_n) = {\lim}_{\cU}\tau_n(a_n).
\]
The set of tracial states that arise in this way is denoted by $\prod T(A_n)/\cU$.
We note that as soon as $\partial T(\prod A_n/\cU)$ is infinite,
the inclusion $\prod T(A_n)/\cU\subset T(\prod A_n/\cU)$ is proper (see \cite{bf}).
Moreover, if we take $A_n$ to be the counterexamples of Pedersen and Petersen (\cite{pp}),
then $\prod T(A_n)/\cU$ (resp.\ $\conv\bigsqcup T(A_n)$)
is not weak$^*$-dense in $T(\prod A_n/\cU)$ (resp.\ $T(\prod A_n)$).
The following theorem is proved by Sato \cite{sato:discrete} (see also \cite{rordam})
in the case where $A$ is a simple nuclear $\mathrm{C}^*$-alge\-bra having finitely many extremal tracial states.

\begin{thm}\label{thm:ultra}
Let $(A_n)_n$ be a sequence of exact $\cZ$-stable $\mathrm{C}^*$-alge\-bras and $\cU$ be
a free ultra\-filter on $\IN$. Then, $\prod T(A_n)/\cU$ (resp.\ $\conv\bigsqcup T(A_n)$)
is weak$^*$-dense in $T(\prod A_n/\cU)$ (resp.\ $T(\prod A_n)$).
In particular, for every $\tau\in T(\prod A_n/\cU)$ and every
separable $\mathrm{C}^*$-sub\-alge\-bra $B\subset\prod A_n/\cU$, there
is $\tau'\in \prod T(A_n)/\cU$ such that $\tau|_B=\tau'|_B$.
\end{thm}

\begin{proof}[Proof of Theorem~\ref{thm:ultra}]
Let $A$ be either $\prod A_n$ or $\prod A_n/\cU$, and denote by $\Sigma\subset T(A)$
either $\conv(\bigsqcup T(A_n))$ or $\prod T(A_n)/\cU$ accordingly.
Suppose that the conclusion of the theorem is false for $\Sigma\subset T(A)$.
Then, by the Hahn--Banach theorem, there are $\tau$
in $T(A)$ and a self-adjoint element $a_0$ in $A$ such that
$\gamma:=\tau(a_0) - \sup_{\sigma\in\Sigma} \sigma(a_0) > 0$.
Let $\alpha = (|\inf_{\sigma\in\Sigma}\sigma(a_0)| - \tau(a_0)) \vee 0$, and
take $b\in A_+$ such that $\tau(b)=\alpha$ and $\|b\|<\alpha+\gamma$.
It follows that $a=a_0+b$ satisfies $\sup_{\sigma\in\Sigma}|\sigma(a)|<\tau(a)$.
Now, expand $a\in A$ as $(a_n)_n$. We may assume that $\|a_n\|\le\|a\|$ for all $n$.
Let $I\in\cU$ (or $I=\IN$ in case $A=\prod A_n$)
be such that $\e_0:=\sup_{n\in I}\sup_{\sigma\in T(A_n)}\sigma(a_n) < \tau(a)$.
Let $R\in\IN$ be such that $\e_1:=\e_0 + C\|a\|R^{-1/2}<\tau(a)$.
Then, by Theorem~\ref{thm:comreg}, for each $n\in I$
there are $b_n(r),c_n(r)\in A_n$ such that
$\sum_{r=1}^R \|b_n(r)\|\|c_n(r)\|\le C\|a\|$ and
$\| a_n - \sum_{r=1}^R [b_n(r),c_n(r)] \|\le\e_1$.
It follows that for $b(r)=(b_n(r))_n$ and $c(r)=(c_n(r))_n\in A$, one has
\[
\tau(a)= \tau(a - \sum_{r=1}^R [b(r),c(r)])\le\| a - \sum_{r=1}^R [b(r),c(r)]\| <\tau(a),
\]
which is a contradiction. This proves the first half of the theorem.

For the second half, let $\tau$ and $B$ be given.
Take a dense sequence $(x(i))_{i=0}^\infty$ in $B$ and expand them as
$x(i)=(x_n(i))_n$.
By the first half, for every $m$, there is $(\tau^{(m)}_n)_n \in \prod T(A_n)$
such that $|\tau(x(i))-\tau^{(m)}_\cU(x(i))|<m^{-1}$ for $i=0,\ldots,m$.
Let
\[
I_m = \{ n\in\IN : |\tau(x(i))-\tau^{(m)}_n(x_n(i))|<m^{-1} \mbox{ for all $i=0,\ldots,m$}\} \in \cU
\]
(so $I_0=\IN$), and $J_m=\bigcap_{l=0}^m I_l \in \cU$.
We define $\tau_n$ to be $\tau^{(m)}_n$ for $n\in J_m\setminus J_{m+1}$.
It is not too hard to check $\tau=\tau_{\cU}$ on $B$.
\end{proof}

In passing, we record the following fact.
\begin{lem}
Let $A$ be a (non-separable) $\mathrm{C}^*$-alge\-bra and $X\subset A$ be a separable subset.
Then there is a separable $\mathrm{C}^*$-sub\-alge\-bra $B\subset A$ that contains $X$
and such that the restriction from $T(A)$ to $T(B)$ is onto.
\end{lem}
\begin{proof}
We may assume that $A$ is unital.
We first claim that for every $x_1,\ldots,x_n \in A$ and $\e>0$,
there is a separable $\mathrm{C}^*$-sub\-alge\-bra $C$ which satisfies the following property:
for every $\tau\in T(C)$ there is $\sigma\in T(A)$ such that
$\max_i|\tau(x_i) - \sigma(x_i)|<\e$.
Indeed if this were not true, then for every $C$ there is $\tau_C\in T(C)$
such that $\max_i|\tau_C(x_i) - \sigma(x_i)|\geq\e$ for all $\sigma\in T(A)$.
The set of separable $\mathrm{C}^*$-sub\-alge\-bras of $A$ is upward directed and
one can find a limit point $\tau$ of $\{\tau_C\}$.
Then, we arrive at a contradiction that $\tau\in T(A)$ satisfies
$\max_i|\tau(x_i) - \sigma(x_i)|\geq\e$ for all $\sigma\in T(A)$.
We next claim that for every separable $\mathrm{C}^*$-sub\-alge\-bra $B_0\subset A$,
there is a separable $\mathrm{C}^*$-sub\-alge\-bra $B_1\subset A$ that contains $B_0$
and such that $\Res_{B_0} T(B_1) = \Res_{B_0} T(A)$ in $T(B_0)$,
where $\Res$ is the restriction map.
Take a dense sequence $x_1,x_2,\ldots$ in $B_0$, and let $C_0=B_0$.
By the previous discussion, there is an increasing sequence of
separable $\mathrm{C}^*$-sub\-alge\-bras $C_0\subset C_1\subset\cdots$ such that
for every $\tau\in T(C_n)$ there is $\sigma\in T(A)$ satisfying
$|\tau(x_i) - \sigma(x_i)|<n^{-1}$ for $i=1,\ldots,n$.
Now, letting $B_1=\overline{\bigcup_n C_n}$ and we are done.
Finally, we iterate this construction and obtain
$X\subset B_0\subset B_1\subset\cdots$ such that
$\Res_{B_n} T(B_{n+1}) = \Res_{B_n} T(A)$.
The separable $\mathrm{C}^*$-sub\-alge\-bra $B=\overline{\bigcup B_n}$ satisfies the desired property.
\end{proof}

Murphy (\cite{murphy}) presents a non-separable example of
a unital non-simple $\mathrm{C}^*$-alge\-bra with a unique faithful tracial state and asks
whether a separable example of such exists. The above lemma answers it.
There is another example, which is moreover nuclear.
Kirchberg (\cite{kirchberg}) proves that the Cuntz algebra $\cO_\infty$
(or any other unital separable exact $\mathrm{C}^*$-alge\-bra) is a subquotient of the CAR algebra $\IM_{2^\infty}$.
Namely, there are $\mathrm{C}^*$-sub\-alge\-bras $J$ and $B$ in $\IM_{2^\infty}$ such that
$J$ is hereditary in $\IM_{2^\infty}$ and is an ideal in $B$ such that $B/J=\cO_\infty$.
It follows that $B$ is a unital separable nuclear non-simple $\mathrm{C}^*$-alge\-bra
with a unique faithful tracial state.

\section{Uniform $2$-norm and the completion}\label{sec:u2}

Recall $S\subset T(A)$, $N=(\bigoplus_{\tau\in S}\pi_\tau)(A)''$, and
the center-valued trace $\ctr\colon N\to\cZ(N)$.
Since $S$ is a closed face of $T(A)$, any normal tracial state
on $N$ restricts to a tracial state on $A$ which belongs to $S$.
Hence, one has
\[
\|a\|_{2,S}=\sup\{ \|a\|_{2,\tau} : \tau\in S\}=\sup\{ \|a\|_{2,\tau} : \tau\in \partial S\}=\|\ctr(a^*a)\|^{1/2}.
\]
Since $S$ is a metrizable closed face of the Choquet simplex $T(A)$, it is also a Choquet simplex and
there is a canonical one-to-one correspondence
\[
S\ni\tau\longleftrightarrow\mu_\tau\in\prob(\partial S),\
\tau(a)=\int \lambda(a)\,d\mu_\tau(\lambda)\mbox{ for $a\in A$}.
\]
By uniqueness of the representing measure $\mu_\tau$, this correspondence is an affine
transformation and extends uniquely to a linear order isomorphism
between their linear spans.

\begin{lem}\label{lem:bmac}
For every $\tau\in S$, there is a normal $*$-iso\-mor\-phism
$\theta_\tau\colon L^\infty(\partial S,\mu_\tau)\to \cZ(\pi_\tau(A)'')$
such that
\[
\tau(\theta_\tau(f)a)=\int f(\lambda)\lambda(a)\,d\mu_\tau(\lambda)
\]
for $a\in A$.
\end{lem}
\begin{proof}
Let $f\in L^\infty(\partial S,\mu_\tau)$ be given.
The right hand side of the claimed equality defines a tracial
linear functional on $A$ whose modulus is dominated by a scalar multiple of $\tau$.
Hence, by Sakai's Radon--Nikodym theorem, there is a unique
$\theta_\tau(f) \in \cZ(\pi_\tau(N))$ that satisfies the claimed equality.
This defines a unital normal positive map $\theta_\tau$ from $L^\infty(\partial S,\mu_\tau)$
into $\cZ(\pi_\tau(A)'')$.
Next, let $z\in\cZ(\pi_\tau(N))_+$ be given. Then, the tracial linear functional $z\tau$ on $A$
defined by $(z\tau)(a)=\tau(az)$ is dominated by $\|z\|\tau$.
Hence one has $\mu_{z\tau}\le\|z\|\mu_\tau$ and $z=\theta_\tau(d\mu_{z\tau}/d\mu_\tau)$ with
$d\mu_{z\tau}/d\mu_\tau\in L^\infty(\partial S,\mu_\tau)$.
This proves $\theta_\tau$ is a positive linear isomorphism such that
$\mu_{\theta_\tau(f)\tau}=f\mu_\tau$. Therefore, one has
$\mu_{\theta_\tau(fg)\tau}=fg\mu_\tau=f\mu_{\theta_\tau(g)\tau}=\mu_{\theta_\tau(f)\theta_\tau(g)\tau}$,
which proves $\theta_\tau(fg)=\theta_\tau(f)\theta_\tau(g)$.
\end{proof}

\begin{proof}[Proof of Theorem~\ref{thm:ast}]
We first find the $*$-homo\-mor\-phism $\theta\colon B(\partial S)\to\cZ(N)$ that satisfies
\[
\tau(\theta(f)a) = \int f(\lambda)\lambda(a)\,d\mu_\tau(\lambda)
\]
for every $a\in A$ and $\tau\in S$, or equivalently, $\pi_\tau(\theta(f))=\theta_\tau(f)$ in $\pi_\tau(A)''$.
For this, it suffices to show that the maps $\theta_\tau|_{B(\partial S)}$,
given in Lemma~\ref{lem:bmac}, are compatible over $\tau\in S$.
We recall that associated with the representation $\pi_\tau$, there is
a unique central projection $p_\tau\in\cZ(N)$ such that $(1-p_\tau) N =\ker\pi_\tau$.
Since $p_\tau\vee p_\sigma = p_{(\tau+\sigma)/2}$,
the family $\{ p_\tau : \tau\in S\}$ is upward directed and $\sup_\tau p_\tau = 1$.
We will show that if $\tau$ and $\sigma$ are such that $\tau\le C\sigma$ for some $C>1$,
then $\theta_\tau(f)=p_\tau \theta_\sigma(f)$ in $\cZ(N)$.
We note that $p_\tau$ is the support projection of $d\tau/d\sigma\in\cZ(N)$.
For every $f\in B(\partial S)$, one has
\begin{align*}
\sigma(\theta_\sigma(\frac{d\mu_\tau}{d\mu_\sigma}f)a)
 &= \int (\frac{d\mu_\tau}{d\mu_\sigma}f)(\lambda)\lambda(a)\,d\mu_\sigma(\lambda)\\
 &= \int f(\lambda)\lambda(a)\,d\mu_\tau(\lambda)\\
 &= \tau(\theta_\tau(f)a)\\
 &= \sigma(\frac{d\tau}{d\sigma}\theta_\tau(f)a).
\end{align*}
This implies $\theta_\sigma(\frac{d\mu_\tau}{d\mu_\sigma}f)=\frac{d\tau}{d\sigma}\theta_\tau(f)$ for every $f$.
In particular, $\theta_\sigma(\frac{d\mu_\tau}{d\mu_\sigma})=\frac{d\tau}{d\sigma}$ and
$p_\tau \theta_\sigma(f)=\theta_\tau(f)$ in $\cZ(N)$.
Therefore, we may glue $\{\theta_\tau\}_{\tau\in S}$ together and obtain a globally defined $*$-homo\-mor\-phism
$\theta\colon B(\partial S)\to\cZ(N)$.
Since $\tau(\theta(\hat{a}))=\int\hat{a}(\lambda)\,d\mu_\tau(\lambda)=\tau(a)$
for every $\tau\in S$, one has $\theta(\hat{a})=\ctr(a)$ for every $a\in A$.
This proves the first part of the theorem.

For the second part, it suffices to prove
\[
\overline{A}^{\st}\supset\{ x\in N : \ctr(xA)\subset \cAff(\partial S),\ \ctr(x^*x)\in\cAff(\partial S)\},
\]
as the converse inclusion is trivial. Take $x$ from the set in the right hand side.
We will prove a stronger assertion that if a net $(b_j)_j$ in $A$ converges
to $x$ ultra\-strongly in $N$, then $x$ is contained in the strict closure of
the convex hull of $\{ b_j : j\}$.
We note that $\Aff(S)\ni f\mapsto f|_{\partial S}\in\cAff(\partial S)$ is an affine order
isomorphism and that every positive norm-one linear functional $\mu$ on $\Aff(S)$ is given
by the evaluation at a point $\tau_\mu\in S$.
(Indeed by the Hahn--Banach theorem, we may regard $\mu$ as a state on $C(S)$,
which is a probability measure on $S$ by the Riesz--Markov theorem. The point
$\tau_\mu=\int \lambda\,d\mu(\lambda)$ satisfies $f(\tau_\mu)=\mu(f)$ for $f\in\Aff(S)$.)
Thus, one has $\ctr((b_j-x)^*(b_j-x))\to0$ weakly in $\cAff(\partial S)$.
Therefore, by the Hahn--Banach theorem, for every $\e>0$ there is a finite
sequence $\alpha_j\geq0$, $\sum\alpha_j=1$ such that $\|\sum_j\alpha_j\ctr((b_j-x)^*(b_j-x))\|<\e$.
By reindexing, we assume $j=1,\ldots,k$.
Let $b=\sum\alpha_jb_j$. We note that
\[
b = \begin{bmatrix}\alpha_1^{1/2} & \cdots &  \alpha_m^{1/2} \end{bmatrix}
\begin{bmatrix}\alpha_1^{1/2}b_1 \\ \vdots \\  \alpha_m^{1/2}b_m \end{bmatrix}
 =: rc.
\]
Hence, $b^*b=c^*r^*rc\le \|r\|^2c^*c=\sum\alpha_jb_j^*b_j$.
It follows that
\begin{align*}
\ctr((b-x)^*(b-x)) &= \ctr(b^*b - b^*x - x^*b + x^*x)\\
 &\le \ctr(\sum\alpha_jb_j^*b_j - \sum\alpha_jb_j^*x - x^*\sum\alpha_jb_j + x^*x)\\
 &= \ctr(\sum\alpha_j(b_j-x)^*(b_j-x))\\
 &< \e.
\end{align*}
This proves the claimed inclusion.
The last assertion will be proved in more general setting as Theorem \ref{thm:cwf}.
\end{proof}

\section{Continuous $\mathrm{W}^*$-bundles}\label{sec:ax}

Let $K$ be a metrizable compact Hausdorff topological space.
We call $M$ a (tracial) \emph{continuous $\mathrm{W}^*$-bundle over $K$}
if the following axiom hold:
\begin{enumerate}[$(1)$]
\item\label{ax:1}
There is a unital positive faithful tracial map $E\colon M\to C(K)$.
\item\label{ax:2}
The closed unit ball of $M$ is complete with respect to
the uniform $2$-norm
\[
\|x\|_{2,\ru}=\|E(x^*x)^{1/2}\|.
\]
\item \label{ax:3}
$C(K)$ is contained in the center of $M$ and $E$ is a conditional expectation.
\end{enumerate}
In case $M$ satisfies only conditions $(\ref{ax:1})$ and $(\ref{ax:2})$, we say
it is a continuous quasi-$\mathrm{W}^*$-bundle.
If we denote by $\pi_E$ the GNS representation of $M$ on the Hilbert $C(K)$-module $L^2(M,E)$,
condition $(\ref{ax:2})$ is equivalent to that $\pi_E(M)$ is strictly closed in $\IB(L^2(M,E))$.
For each point $\lambda\in K$, we denote by $\pi_\lambda$ the GNS representation
for the tracial state $\tau_\lambda:=\ev_\lambda\circ E$,
and also $\|x\|_{2,\lambda}=\tau_\lambda(x^*x)^{1/2}$.
We call each $\pi_\lambda(M)$ a \emph{fiber} of $M$.
A caveat is in order: the system $(M,K,\pi_\lambda(M))$ need not be a continuous $\mathrm{C}^*$-bundle
because $\ker\pi_\lambda$ may not coincide with $C_0(K\setminus\{\lambda\})M$---rather it coincides
with the strict closure of that.
In particular, for $x\in M$, the map $\lambda\mapsto\|\pi_\lambda(x)\|$
need not be upper semi-continuous (but it is lower semi-continuous).
The strict completion $\overline{A}^{\st}$ studied in Section \ref{sec:u2}
is a continuous quasi-$\mathrm{W}^*$-bundle over $S$, and
by Theorem \ref{thm:ast}, it is a continuous $\mathrm{W}^*$-bundle over $\partial S$
if $\partial S$ is closed in $S$.
Conversely, if each fiber $\pi_\lambda(M)$ is a factor, then $K$ can be viewed as
a closed subset of the extreme boundary of $T(M)$ and hence
the closed convex hull $S$ of $K$ is a metrizable closed face of $T(M)$
such that $\partial S=K$.

\begin{thm}\label{thm:cwf}
Let $M$ be a continuous $\mathrm{W}^*$-bundle over $K$.
Then, $\pi_\lambda(M)=\pi_\lambda(M)''$ for every $\lambda\in K$.
Moreover, if a bounded function $f\colon K\ni\lambda\mapsto f(\lambda)\in\pi_\lambda(M)$
is continuous in the following sense:
for every $\lambda_0\in K$ and $\e>0$, there are a neighborhood $O$ of $\lambda_0$
and $c\in M$ such that
\[
\sup_{\lambda\in O}\| \pi_\lambda(c)-f(\lambda) \|_{2,\lambda}<\e;
\]
then there is $a\in M$ such that $\pi_\lambda(a)=f(\lambda)$.
\end{thm}

\begin{proof}
Let $\lambda\in K$ be given. By Pedersen's up-down theorem (Theorem 2.4.4 in \cite{pedersen}),
it suffices to show that
$\pi_\lambda(M)$ is closed in $\pi_\lambda(M)''$ under monotone sequential limits.
Let $(x_n)_{n=0}^\infty$ be an increasing sequence in $\pi_\lambda(M)_+$
such that $x_n \nearrow x$ in $\pi_\lambda(M)''$.
We may assume that $\|x_n-x\|_{2,\lambda}<2^{-n}$.
We lift $(x_n)_{n=0}^\infty$ to an increasing sequence $(a_n)_{n=0}^\infty$ in $M$
such that $a_n\le \|x\|+1$.
Let $b_n=a_n-a_{n-1}$ for $n\geq1$.
Since $\tau_\lambda(b_n^*b_n)<4^{-n+2}$,
there is $f_n\in C(K)_+$ such that $0\le f_n\le 1$, $f_n(\lambda)=1$, and $\| E(b_n^*b_n)f_n^2\|\le 4^{-n+2}$.
It follows that the series $a_0+\sum_{n=1}^\infty b_nf_n$ is convergent in the uniform $2$-norm.
Moreover, since $a_0+\sum_{k=1}^n b_kf_k\le a_0+\sum_{k=1}^n b_k=a_n\le \|x\|+1$,
the series is norm bounded. Therefore, the series converges in $M$, by the completeness of the closed
unit ball of $M$. The limit point $a$ satisfies $\pi_\lambda(a) = x$.

We prove the second half.
Let us fix $n$ for a while.
For each $\lambda$, there is $b_\lambda\in M$ such that
$\|b_\lambda\|\le\|f(\lambda)\|$ and $\pi_\lambda(b_\lambda)=f(\lambda)$.
By continuity, there is a neighborhood
$O_\lambda$ of $\lambda$ such that
$\|\pi_{\tau}(b_\lambda) - f(\tau)\|_{2,\tau}<n^{-1}$
for $\tau\in O_\lambda$. Since $K$ is compact, it is covered by a finite family
$\{ O_{\lambda_i}\}$. Let $g_i\in C(K)\subset\cZ(M)$ be a partition
of unity subordinated by it. Then, $a_n:=\sum_i g_i b_{\lambda_i}\in M$
satisfies $\|a_n\|\le\|f\|_\infty$ and $\sup_\tau\|\pi_{\tau}(a_n) - f(\tau)\|_{2,\tau}<n^{-1}$.
It follows that $(a_n)$ is a norm bounded and Cauchy in the uniform $2$-norm.
Hence it converges to $a\in M$ such that $\pi_\lambda(a)=f(\lambda)$ for every $\lambda\in K$.
\end{proof}

The following is a $\mathrm{W}^*$-analogue of the result for $\mathrm{C}^*$-alge\-bras in \cite{hrw},
and is essentially the same as Proposition 7.7 in \cite{kr}.

\begin{cor}\label{cor:mcduff}
Let $M$ be a continuous $\mathrm{W}^*$-bundle over $K$.
Assume that each fiber $\pi_\lambda(M)$ has the McDuff property and
that $K$ has finite covering dimension.
Then, for every $k$, there is an approximately central approximately multiplicative
embedding of $\IM_k$ into $M$, namely
a net of unital completely positive maps $\p_n\colon\IM_k\to M$
such that $\limsup_n\|\p_n(xy) - \p_n(x)\p_n(y)\|_{2,\ru}=0$
and $\limsup_n\|[\p_n(x),a]\|_{2,\ru}=0$ for every $x,y\in\IM_k$ and $a\in M$.
\end{cor}
\begin{proof}
The proof is particularly easy when $K$ is zero-dimensional:
Since $\pi_\lambda(M)$ is McDuff, there is an approximately central embedding
of $\IM_k$ into $\pi_\lambda(M)$.
We lift it to a unital completely positive map $\psi_\lambda\colon\IM_k\to M$.
It is almost multiplicative on a neighborhood $O_\lambda$ of $\lambda$.
Since $K$ is compact and zero-dimensional, there is a partition of $K$ into
finitely many clopen subsets $\{V_i\}$ such that $V_i\subset O_{\lambda_i}$.
By Theorem \ref{thm:cwf}, one can define $\p\colon\IM_k\to M$ by the relation
$\pi_\lambda\circ\p=\pi_\lambda\circ\psi_{\lambda_i}$ for $\lambda\in V_i$.
The case $0<\dim K<+\infty$ is more complicated but follows from
a standard argument involving order-zero maps.
See Section 7 in \cite{kr} (or \cite{sato,tww}) for the detail.
\end{proof}

Every separable hyperfinite von Neumann algebra with a faithful normal tracial state
has a trace preserving embedding into the separable hyperfinite $\mathrm{II}_1$ factor $\cR$.
We consider coordinatization of such embeddings for strictly separable fiberwise hyperfinite
continuous quasi-$\mathrm{W}^*$-bundle.
We define the $\mathrm{C}^*$-alge\-bra $C_\sigma(K,\cR)$ to be the subalgebra of $\ell_\infty(K,\cR)$ which
consists of those norm-bounded functions $f\colon K\to\cR$
that are continuous from $K$ into $L^2(\cR,\tau_{\cR})$.

\begin{thm}\label{thm:emb}
Let $M$ be a strictly separable continuous quasi-$\mathrm{W}^*$-bundle over $K$
such that $\pi_\lambda(M)''$ is hyperfinite for every $\lambda\in K$.
Then, there are an embedding $\theta\colon M\hookrightarrow C_\sigma(K,\cR)$
and embeddings $\iota_\lambda\colon\pi_\lambda(M)\hookrightarrow\cR$
such that $\ev_\lambda\circ\theta = \iota_\lambda\circ\pi_\lambda$.
If $M$ is moreover a continuous $\mathrm{W}^*$-bundle, then one has
\[
\theta(M) = \{ f\in C_\sigma(K,\cR) : f(\lambda)\in(\iota_\lambda\circ\pi_\lambda)(M)''\}.
\]
\end{thm}

Recall the fact that if $(A,\tau)$ is a separable hyperfinite
von Neumann algebra with a distinguished tracial state, then
a trace-preserving embedding of $A$ into the tracial ultra\-power $\cR^\omega$
of the hyperfinite $\mathrm{II}_1$ factor is unique up to unitary conjugacy
(see \cite{jung}). For every $n$-tuples $x_1,\ldots,x_n \in P$
and $y_1,\ldots,y_n \in Q$ in hyperfinite $\mathrm{II}_1$ factors $P$ and $Q$,
we define
\[
d(\{x_i\}_{i=1}^n,\{y_i\}_{i=1}^n) = \inf_{\pi,\rho}\max_i\|\pi(x_i)-\rho(y_i)\|_2,
\]
where the infimum runs over all trace-preserving embeddings
of $P$ and $Q$ into $\cR^\omega$.
Then, $d$ is a pseudo-metric and it depends on $(W^*(\{x_1,\ldots,x_n\}),\tau)$,
i.e., the joint distribution of $\{x_1,\ldots,x_n\}$ with respect to $\tau_P$,
rather than the specific embedding of $W^*(\{x_1,\ldots,x_n\})$ into $P$.
Once $*$-iso\-mor\-phisms $P\cong Q\cong \cR$ are fixed,
$P$ and $Q$ are embedded into $\cR^\omega$ as constant sequences and
\[
d(\{x_i\}_{i=1}^n,\{y_i\}_{i=1}^n) = \inf_{U\in\cU(\cR^\omega)}\max_i\|\Ad_U(x_i)-y_i\|_2.
\]
It follows that
\[
d(\{x_i\}_{i=1}^n,\{y_i\}_{i=1}^n) = \inf_\pi \max_i \| \pi(x_i) - y_i \|_2,
\]
where infimum runs over all trace-preserving $*$-homo\-mor\-phisms $\pi$ from $W^*(\{x_1,\ldots,x_n\})$ into $Q$,
or over all $*$-iso\-mor\-phisms $\pi$ from $P$ onto $Q$.
If $M$ is a continuous quasi-$\mathrm{W}^*$-bundle, then for every $a_1,\ldots,a_n\in M$, the map
\[
K\ni\lambda\mapsto\{\pi_\lambda(a_i)\}_{i=1}^n
\]
is continuous with respect to $d$.

\begin{lem}\label{lem:extension}
Let $N=C_\sigma(K,\cR)$ or any other continuous $\mathrm{W}^*$-bundle
over $K$ such that $\ev_\lambda(N)\cong\cR$ for every $\lambda\in K$
and such that for every $k\in\IN$ there is an approximately central approximately
multiplicative embedding of $\IM_k$ into $N$.
Let $M$ be a continuous quasi-$\mathrm{W}^*$-bundle over $K$ such that
$\pi_\lambda(M)''$ is hyperfinite for every $\lambda\in K$, and
let $F_0\subset F_1$ be finite subsets in the unit ball of $M$ and $\e>0$.
Assume that there is a map $\theta_0$ from $F_0$ into the unit ball of $N$ such that
\[
\sup_{\lambda\in K} d(\{\pi_\lambda(a)\}_{a\in F_0},\{\ev_\lambda(\theta_0(a))\}_{a\in F_0})<\e.
\]
Then, for every $\delta>0$, there is a map $\theta_1$ from $F_1$ into the unit ball
of $N$ such that
\[
\sup_{\lambda\in K} d(\{\pi_\lambda(a)\}_{a\in F_1},\{\ev_\lambda(\theta_1(a))\}_{a\in F_1})<\delta
\]
and
\[
\max_{a\in F_0} \|\theta_1(a) - \theta_0(a)\|_{2,\ru}<\e.
\]
\end{lem}
Here the symbol $\ev_\lambda$, instead of $\pi_\lambda$, is used for the $N$ side
to make a distinction from the $M$ side.
\begin{proof}
For each $\lambda$, there is a trace-preserving embedding
$\rho_\lambda\colon\pi_\lambda(M)\to\ev_\lambda(N)$.
By the remarks preceding this lemma, we may assume that
\[
\max_{a\in F_0}\|\rho_\lambda(\pi_\lambda(a)) - \ev_\lambda(\theta_0(a))\|_2<\e.
\]
For each $a\in F_1$, we lift $(\rho_\lambda\circ\pi_\lambda)(a)\in\ev_\lambda(N)$
to $a^\lambda\in N$ with $\|a^\lambda\|\le1$.
There is a neighborhood $O_\lambda$ of $\lambda$ such that
$\tau\in O_\lambda$ implies
\[
d(\{\pi_\tau(a)\}_{a\in F_1},\{\ev_\tau(a^\lambda)\}_{a\in F_1})<\delta
\]
and
\[
\max_{a\in F_0}\| \ev_\tau(a^\lambda) - \ev_\tau(\theta_0(a)) \|_2<\e.
\]
By compactness, $K$ is covered by a finite family $\{ O_{\lambda_j}\}$.
Take a partition of unity $g_j\in C(K)$ subordinated by $\{ O_{\lambda_j}\}$.
Let $h_0=0$ and $h_j=\sum_{i=1}^{j} g_i$.
For each $k$, take an approximately central approximately
multiplicative embedding $\p_{k,n}$ of $\IM_k$ into $N$.
Since the closed unit ball of $\IM_k$ is norm-compact, one has
\[
\forall a\in N\quad\limsup_n\,\sup\{ \|[\p_{k,n}(x),a]\|_{2,\ru} : x\in\IM_k,\,\|x\|\le1\}=0.
\]
For $t\in[0,1]$, we define $p_t\in\IM_k$ to be
$\mathrm{diag}(1,\ldots,1,t-\lfloor t\rfloor,0,\ldots,0)$,
with $1$s in the first $\lfloor t\rfloor$ diagonal entries,
$t-\lfloor t\rfloor$ in the $(\lfloor t\rfloor+1)$-th entry, and $0$s in the rest.
It follows that $t\mapsto p_t$ is continuous, $0\le p_t\le 1$, $\tr(p_t)=t$,
and $\tau(p_t-p_t^2)\le(4k)^{-1}$. We write $p_{[s,t]}=p_t-p_s$.
With the help of Theorem \ref{thm:cwf}, we define $f_{k,n,j}\in N$ to be the element such that
\[
\ev_\lambda(f_{k,n,j}) = \ev_\lambda(\p_{k,n}( p_{[h_{j-1}(\lambda),h_j(\lambda)]} )).
\]
For $a\in F_1$, we define $\theta^{k,n}_1(a)\in N$ by
$\theta^{k,n}_1(a) = \sum_j f_{k,n,j}^{1/2}a^{\lambda_j}f_{k,n,j}^{1/2}$.
Since $F_1':=F_1\cup\{a^{\lambda_j} : a\in F_1,\,j\}$ is finite,
it is not too hard to see
\begin{align*}
\limsup_k\limsup_n\max_{a\in F_0} \|\theta^{k,n}_1(a) - \theta_0(a)\|_{2,\ru}
< \e.
\end{align*}
It remains to estimate
\[
d(\{\pi_\tau(a)\}_{a\in F_1},\{\ev_\tau(\theta^{k,n}_1(a))\}_{a\in F_1}).
\]
Let $k$ be fixed for the moment.
Since $(\p_{k,n})_n$ is approximately multiplicative, there are unital $*$-homo\-mor\-phisms
$\psi_{k,n}^\tau\colon \IM_k\to\ev_\tau(N)$ such that
\[
\limsup_n\sup_\tau\sup_{x\in\IM_k,\,\|x\|\le1}\|\ev_\tau\circ\p_{k,n}(x) - \psi_{k,n}^\tau (x)\|_2=0.
\]
Let $E_{k,n}^\tau$ be the trace-preserving
conditional expectation from $\ev_\tau(N)$ onto the relative commutant
$\psi_{k,n}^\tau(\IM_k)'\cap\ev_\tau(N)$,
which is given by $E_{k,n}^\tau(b)=|G|^{-1}\sum_{u\in G} \psi_{k,n}^\tau(u) b \psi_{k,n}^\tau(u)^*$
for the group $G$ of permutation matrices in $\cU(\IM_k)$.
It follows that
\[
\limsup_n\sup_\tau\| \ev_\tau(b) - E_{k,n}^\tau(\ev_\tau(b))\|_2=0
\]
for every $b\in N$. This implies
\[
\limsup_n \sup_{j,\,\tau\in O_{\lambda_j}}
d(\{\pi_\tau(a)\}_{a\in F_1},\{E_{k,n}^\tau(\ev_\tau(a^{\lambda_j}))\}_{a\in F_1})<\delta,
\]
\[
\limsup_n \sup_{j,\,\tau\in O_{\lambda_j}}
d(\{\ev_\tau(a^{\lambda_j})\}_{a\in F_1},\{E_{k,n}^\tau(\ev_\tau(a^{\lambda_j}))\}_{a\in F_1})=0,
\]
and
\[
\limsup_n \sup_{j,\,\tau\in O_{\lambda_j}}
d(\{\ev_\tau(\theta^{k,n}_1(a))\}_{a\in F_1},
\{\sum_j \psi_{k,n}^\tau(p_{[h_{j-1}(\lambda),h_j(\lambda)]}) E_{k,n}^\tau(\ev_\tau(a^{\lambda_j}))\}_{a\in F_1})=0.
\]
If we view $\ev_\tau(N)=\IM_k(\psi_{k,n}^\tau(\IM_k)'\cap\ev_\tau(N))$,
then $a'=E_{k,n}^\tau(\ev_\tau(a))$
looks like $\mathrm{diag}(a',a',\ldots,a')$,
and $\psi_{k,n}^\tau(p_t)$ looks like $\mathrm{diag}(1,\ldots,1,t-\lfloor t\rfloor,0\ldots,0)$.
Hence, one has
\begin{align*}
&\sup_\tau d(\{\pi_\tau(a)\}_{a\in F_1},
 \{\sum_j \psi_{k,n}^\tau(p_{[h_{j-1}(\lambda),h_j(\lambda)]}) E_{k,n}^\tau(\ev_\tau(a^{\lambda_j}))\}_{a\in F_1})^2\\
&\quad< \frac{2|\{O_{\lambda_j}\}|}{k}+\sum_j g_j(\tau) d(\{\pi_\tau(a)\}_{a\in F_1},
  \{E_{k,n}^\tau(\ev_\tau(a^{\lambda_j}))\}_{a\in F_1})^2.
\end{align*}
Altogether, one has
\[
\limsup_k\limsup_n \sup_\tau d(\{\pi_\tau(a)\}_{a\in F_1},\{\ev_\tau(\theta^{k,n}_1(a))\}_{a\in F_1})<\delta.
\]
Therefore, for some $k,n$, the map $\theta_1=\theta^{k,n}_1$ satisfies the desired properties.
\end{proof}

\begin{proof}[Proof of Theorem \ref{thm:emb}]
Let $(a_n)_{n=1}^\infty$ be a strictly dense sequence in the unit ball of $M$.
We use Lemma \ref{lem:extension} recursively and obtain sequences
$(\{\theta_n(a_i)\}_{i=1}^n)_{n=1}^\infty$ in $C_\sigma(K,\cR)$ such that
\[
\sup_\lambda d(\{\ev_\lambda(\theta_n(a_i))\}_{i=1}^n, \{\pi_\lambda(a_i)\}_{i=1}^n)<2^{-n}
\]
and
\[
\max_{i=1,\ldots,n-1}\|\theta_n(a_i) - \theta_{n-1}(a_i)\|_{2,\ru}<2^{-(n-1)}.
\]
Then, each sequence $(\theta_n(a_i))_{n=i}^\infty$ converges to an element $\theta(a_i)\in C_\sigma(K,\cR)$.
The map $\theta$ extends to a $*$-homo\-mor\-phism from $M$ into $C_\sigma(K,\cR)$, and
$\ev_\lambda\circ\theta$ factors through $\pi_\lambda$. This proves the first assertion.
The second follows from Theorem \ref{thm:cwf}.
\end{proof}

We give a criterion for a continuous $\cR$-bundle
to be a trivial bundle.

\begin{thm}\label{thm:crf}
Let $M$ be a strictly separable continuous $\mathrm{W}^*$-bundle over
$K$ such that $\pi_\lambda(M)\cong\cR$ for every $\lambda\in K$.
Then, the following are equivalent.
\begin{enumerate}[$(\rm i)$]
\item\label{con:tfld}
$M\cong C_\sigma(K,\cR)$ as a continuous $\mathrm{W}^*$-bundle.
\item\label{con:ctrs1}
There is a sequence $(p_n)_n$ in $M$ such that $0\le p_n\le1$,
$\| p_n - p_n^2\|_{2,\ru}\to0$,
$\|E(p_n)-1/2\|\to0$, and $\|[ p_n, a ]\|_{2,\ru}\to0$ for all $a\in M$.
\item\label{con:ctrs2}
For every $k$, there is an approximately central approximately multiplicative
embedding of $\IM_k$ into $M$.
\end{enumerate}
\end{thm}

\begin{proof}
The implication $(\mathrm{\ref{con:tfld}})\Rightarrow(\mathrm{\ref{con:ctrs1}})$ is obvious.
For $(\mathrm{\ref{con:ctrs1}})\Rightarrow(\mathrm{\ref{con:ctrs2}})$, we
we observe that since $\pi_\lambda(M)$'s are all factors,
the central sequence $(p_n)_n$ satisfies $\|E(p_na)-E(p_n)E(a)\|\to0$ for every $a\in M$.
Indeed, let $a\in M$ and $\e>0$ be given.
By the Dixmier approximation theorem and the proof of Theorem \ref{thm:ast},
there are $u_1,\ldots,u_k\in \cU(M)$
such that $\|E(a) - \frac{1}{k}\sum_{i=1}^k u_iau_i^*\|_{2,\ru}<\e$. It follows that
\begin{align*}
\limsup_{n\to\infty}\|E(p_n)E(a) - E(p_na)\|
 &= \limsup_{n\to\infty} \| E(p_nE(a)) - \frac{1}{k}\sum_{i=1}^k E(u_ip_nau_i^*) \|\\
 &= \limsup_{n\to\infty} \|E(p_n(E(a) - \frac{1}{k}\sum_{i=1}^ku_iau_i^*))\|\\
 &<\e.
\end{align*}
Let $m\in\IN$ be arbitrary. For a given finite sequence $(p_n)_{n=1}^m$, $0\le p_i\le 1$,
and $\nu\in\{0,1\}^m$, we define $q_\nu\in M$ by
\[
q_\nu=r_1^{1/2}\cdots r_{m-1}^{1/2} r_m r_{m-1}^{1/2}\cdots r_1^{1/2} \in M,
\]
where $r_i=p_i$ or $1-p_i$ depending on $\nu(i)\in\{0,1\}$.
We note that $q_\nu\geq0$ and $\sum q_\nu=1$.
By choosing $(p_n)_{n=1}^m$ appropriately, we obtain an approximately central
approximately multiplicative embedding of $\ell_\infty(\{0,1\}^m)$ into $M$.
Now, condition $(\mathrm{\ref{con:ctrs2}})$ follows by choosing at the local level
approximately central approximately multiplicative embeddings
of $\IM_k$ into $\pi_\lambda(M)$ and glue them together, as in the proof of Lemma \ref{lem:extension},
by an approximately central approximately projective partition of unity.

The proof of $(\mathrm{\ref{con:ctrs2}})\Rightarrow(\mathrm{\ref{con:tfld}})$
is similar to that of Theorem \ref{thm:emb}.
Let $(a_n)_{n=1}^\infty$ (resp.\ $(b_n)_{n=1}^\infty$) be a strictly dense sequence
in the unit ball of $M$ (resp.\ $C_\sigma(K,\cR)$).
We recursively construct finite subsets $F_1\subset F_2\subset\cdots$ of $M$
and maps $\theta_n\colon F_n\to C_\sigma(K,\cR)$ such that
$\{a_1,\ldots,a_n\}\subset F_n$,
\[
\sup_\lambda d(\{\ev_\lambda(\theta_n(a))\}_{a\in F_n}, \{\pi_\lambda(a)\}_{a\in F_n})<2^{-n},
\]
\[
\max_{a\in F_{n-1}}\|\theta_n(a) - \theta_{n-1}(a)\|_{2,\ru}<2^{-(n-1)},
\]
and $\{b_1,\ldots,b_n\}\subset\theta_n(F_n)$.
Let $F_0=\emptyset$ and suppose that we have constructed up to $n-1$.
Let $F_n'=F_{n-1}\cup\{a_n\}$.
We use Lemma \ref{lem:extension} and obtain a map $\theta_n'\colon F_n'\to C_\sigma(K,\cR)$
such that
\[
\sup_\lambda d(\{\ev_\lambda(\theta_n'(a))\}_{a\in F_n'}, \{\pi_\lambda(a)\}_{a\in F_n'})<2^{-(n+1)}
\]
and
\[
\max_{a\in F_{n-1}}\|\theta_n'(a) - \theta_{n-1}(a)\|_{2,\ru}<2^{-(n-1)}.
\]
We may assume that $\theta_n'$ is injective and $\theta_n'(F_n')$ does
not contain any of $b_1,\ldots,b_n$.
We use Lemma \ref{lem:extension} again but this time to
$\theta_n'(F_n')\subset \tilde{F}:=\theta_n'(F_n')\cup\{b_1,\ldots,b_n\}$ and $(\theta_n')^{-1}$.
Then, there is $\psi\colon \tilde{F}\to M$ such that
\[
\sup_\lambda d(\{\pi_\lambda(\psi(b))\}_{b\in\tilde{F}},\{\ev_\lambda(b)\}_{b\in\tilde{F}})<2^{-(n+1)}
\]
and
\[
\max_{a\in F_n'}\| a - \psi(\theta_n'(a))\|_{2,\ru}<2^{-(n+1)}.
\]
Now, we set $F_{n}=F_n'\cup\{\psi(b_1),\ldots\psi(b_n)\}$
(which can be assumed to be a disjoint union)
and define $\theta_{n}\colon F_{n}\to C_\sigma(K,\cR)$ by
$\theta_{n}=\theta_n'$ on $F_n'$ and $\theta_{n}(\psi(b_i))=b_i$.
One has
\begin{align*}
&\sup_\lambda d(\{\ev_\lambda(\theta_n(a))\}_{a\in F_n}, \{\pi_\lambda(a)\}_{a\in F_n}) \\
 &\quad\le\sup_\lambda (d(\{\ev_\lambda(b)\}_{b\in\tilde{F}}, \{\pi_\lambda(\psi(b))\}_{b\in\tilde{F}})
  +\max_{a\in F_n'}\| \pi_\lambda(\psi(\theta_n'(a))) - \pi_\lambda(a) \|_{2}) \\
 &\quad< 2^{-n}
\end{align*}
as desired. By taking the limit of $(\theta_n)_n$, one obtains a $*$-iso\-mor\-phism
$\theta$ from $M$ onto $C_\sigma(K,\cR)$.
\end{proof}

By combining Corollary \ref{cor:mcduff} and Theorem \ref{thm:crf}, one obtains
the following $\mathrm{W}^*$-analogue of Theorem 1.1 in \cite{dw}.
This also implies Theorem \ref{thm:stct}.
It is unclear whether the finite-dimensionality assumption is essential.
\begin{cor}
Let $M$ be a strictly separable continuous $\mathrm{W}^*$-bundle over $K$.
If every fiber $\pi_\lambda(M)$ is isomorphic to $\cR$ and $K$ has
finite covering dimension, then $M\cong C_\sigma(K,\cR)$ as a continuous $\mathrm{W}^*$-bundle.
\end{cor}

\end{document}